\documentclass[12pt,MSc]{muthesis_2023}

\usepackage{amsmath}
\usepackage{amsfonts}
\usepackage{amssymb}
\usepackage{lipsum}
\usepackage[utf8]{inputenc}
\usepackage{graphicx}
\usepackage{tikz}

\usepackage{amsthm}
\usepackage{xpatch}
\xpatchcmd\swappedhead{~}{.~}{}{}
\swapnumbers 

\theoremstyle{plain}
\newtheorem{FactCounter}{dummy}[section]
\newtheorem{Theorem}[FactCounter]{Theorem} %
\newtheorem{Lemma}[FactCounter]{Lemma} %
\newtheorem{Corollary}[FactCounter]{Corollary} %
\newtheorem{Conjecture}[FactCounter]{Conjecture} %
\theoremstyle{definition}
\newtheorem{Definition}[FactCounter]{Definition} %
\theoremstyle{remark}
\newtheorem{Remark}[FactCounter]{Remark} %
\newtheorem{Notation}[FactCounter]{Notation} %
\newtheorem{Example}[FactCounter]{Example} %

\newtheoremstyle{LayoutVoid}
  {1ex}
  {0ex}
  {\normalfont}
  {}
  {\bf}
  {.}
  {1ex}
  {}
\theoremstyle{LayoutVoid}

\numberwithin{equation}{section}

\DeclareMathOperator{\tr}{tr}
\DeclareMathOperator{\ad}{ad}
\DeclareMathOperator{\gl}{\mathfrak{gl}}
\DeclareMathOperator{\End}{End}
\DeclareMathOperator{\rad}{rad}
\DeclareMathOperator{\M}{M}
\DeclareMathOperator{\fsl}{\mathfrak{fsl}}
\DeclareMathOperator{\Sym}{Sym}
\DeclareMathOperator{\charr}{char}

\usepackage{cite}

\usepackage{imakeidx}
\makeindex

\usepackage{listings}
\usepackage{hyperref}

\begin{document}

\title{Lie Theory Theorems over Positive Characteristic and Modular Lie Algebras}

\author{Eun H. Park}

\school{Mathematics}
\faculty{Science and Engineering}

\beforeabstract
Sometimes, it is very important to consider what type of setting is assumed when
studying a mathematical object. For example, in Galois theory, properties can com-
pletely change if we study a field extension over $F_p$ instead of a field over $\mathbb{Q}$. When we
consider base fields for modules, algebras, or vector spaces, we often recall commonly
used fields such as $\mathbb{C}$ and fields $F$ with $\charr F= p$.

Similar behavior arises in the study of Lie algebras. Properties that hold for Lie
algebras over a field of characteristic zero do not necessarily hold over a field of characteristic $p$. In general, we are more familiar with
studying Lie algebras and their representations over $\mathbb{C}$. However, an interesting fact
is that new properties can be discovered by studying the theory over fields of positive
characteristic.

Therefore, we will closely examine how theorems and properties in Lie algebra
theory do not hold or behave differently when the base field has characteristic p. In
fact, there is a related area of study known as \textit{modular Lie theory} that deals specifically
with this setting. In this theory, we study concepts such as the definition of restricted
Lie algebras, that is, Lie algebras $L$ equipped with a p-mapping $[p] : L \rightarrow L$, defined
as $x\mapsto x^{[p]}$, where the base field has prime characteristic. In other words, the theory
introduces a new tool—the $p$-mapping—for the study of modular Lie algebras.

In this project, we aim to study Lie algebras defined over fields of positive characteristic. Specifically, the main focus will be on how Lie algebras behave over such fields
and how we can develop the general framework of modular Lie theory based on the
insights and structures that arise in this setting.

\prefacesection{Acknowledgements}
I would like to express my sincere gratitude to Dr. David Stewart for suggesting this topic for my main project and for his guidance and support throughout.

\afterpreface


\chapter{Introduction}

This project investigates how fundamental properties and theorems in Lie theory are affected when the base field has prime characteristic.

We will briefly observe that the structural features of Lie algebras are closely tied to the characteristic of the underlying base field \( F \).

In Chapter 2, only the essential definitions and basic concepts of Lie algebras will be introduced for the main project. For example, we will define Lie algebras, their representations, adjoint representations, solvable Lie algebras, and, lastly, the Killing form. These concepts are important in Lie theory and also in modular Lie theory.

Since we are dealing with base fields of characteristic \( p \), it is important to understand why many theorems in Lie theory assume that the base field is algebraically closed. In fact, an algebraically closed field guarantees that all eigenvalues lie in the field, which is a fundamental aspect of linear algebra. In Chapter 3, we will discuss why the base field is assumed to be algebraically closed in Lie theory.

\(\mathfrak{sl}(2,F)\) is not only an important Lie algebra in classical Lie theory but also in modular Lie theory. Therefore, in Chapter 4, \(\mathfrak{sl}(2,F)\) will be introduced. Furthermore, \(\fsl(2,F)\), called the \textit{fake special linear Lie algebra}, will also be introduced in this chapter. \(\fsl(2,F)\) is important for this project in terms of proving that many theorems in Lie theory do not hold over a field of prime characteristic.

In Chapter 5, we will introduce how Lie theory theorems become false when the base field has characteristic \( p \). Important theorems in Lie theory—such as Lie's theorem, Cartan's criterion, the Weyl character formula, the Jordan decomposition, and representations of Lie algebras—do not hold in this setting and have counterexamples. While introducing counterexamples for each theorem, we will also explore how certain conditions in the theorems can be changed to make them valid in a field of prime characteristic. Therefore, this chapter will show why new tools are needed to study modular Lie algebras.

In Chapter 6, we introduce modular Lie algebras. We define the \( p \)-mapping for modular Lie algebras. To define the \( p \)-mapping, we first construct a framework that makes sense in modular Lie theory. After this discussion, we define \( p \)-mappings and restricted Lie algebras. Moreover, examples of restricted Lie algebras will be introduced to develop intuition for this concept.

In Chapter 7, we will conclude the project with a general discussion of modular Lie algebras and present the insights gained from this investigation.

\chapter{Lie Algebras}
\chaptermark{Lie Algebras} 

The general assumption throughout this project is that all vector spaces over any field considered are finite-dimensional.

\section[Definition of Lie Algebras]{Definition of Lie Algebra}
A Lie algebra $L$ is a vector space over a (base) field $F$ equipped with a Lie multiplication $f: L \times L \rightarrow L$. The definition of Lie algebra is stated in  \ref{Liedef}.

\begin{Notation}
A \textit{$F$-vector space}\index{F-vector space} is a vector space over $F$.
\end{Notation}

\begin{Definition}
Let $U, V$ and $W$ be $F$-vector spaces.
A map $f: U\times V \rightarrow W$ is called \textit{bilinear map}\index{bilinear map} if $f(ax_1+bx_2, y)= af(x_1,y)+bf(x_2,y)$ and $f(x, ay_1+by_2)=af(x, y_1)+bf(x, y_2)$ for all $a, b \in F$, $x_1, x_2 \in U$ and $y_1, y_2 \in V$.

\end{Definition}

Our definition is equivalent to the definition given in \cite[Def 2.1.2]{EW}.

\begin{Definition}\label{Liedef}
A \textit{Lie algebra}\index{Lie algebra} is a $F$-vector space $L$ defined with a bilinear map $f: L\times L \rightarrow L$ satisfying the following properties:
\begin{equation} \label{L1}
f(x, x)=0 \;\;\; \text{for all } x \in L
\end{equation}
\begin{equation} \label{Jacid}
f(f(x, y), z) + f(f(y, z), x) + f(f(z, x), y) = 0\;\;\; \text{for all} \; x, y, z \in L.
\end{equation}

The equation \ref{Jacid} is called \textit{Jacobi Identity}\label{Jacobi identity}.

A bilinear $f$ satisfying \ref{L1} and \ref{Jacid} is called \textit{Lie multiplication}\label{Lie multiplication}. 

\end{Definition}

\begin{Remark}
A Lie algebra $L$ is a pair $(L, f)$ where $f$ is Lie multiplication.
\end{Remark}

\begin{Example}\label{matlie}

Consider the vector space \( \M(n,F) \) of \( n \times n \) matrices over a field \( F \). If we define a product 
\[
[\;,]: \M(n,F) \times \M(n,F) \rightarrow \M(n,F)
\]
as
\[
[X, Y] = XY - YX \quad \text{for all } X, Y \in \M(n,F),
\]
then this map is a Lie bracket.

First, by the linearity of matrix multiplication, this map is bilinear. Clearly, 
\[
[X, X] = X^2 - X^2 = 0 \quad \text{for all } X \in \M(n,F).
\]
For the Jacobi identity in \ref{Jacid}, we apply the defined product \( [\;,] \) on the left-hand side. Then, 
$$(XY - YX)Z-Z(XY - YX) + (YZ - ZY)X - X(YZ - ZY) + (ZX - XZ)Y - Y(ZX - XZ)$$
becomes
$$XYZ-YXZ-ZXY + ZYX + YZX - ZYX - XYZ + XZY+ ZXY - XZY - YZX + YXZ$$
$$ = 0.$$
This shows that the bilinear map \( [\;,] \) satisfies the Jacobi identity.

Therefore, \( \M(n,F) \) equipped with \( [\;,] \) is a Lie algebra.

\end{Example}

\begin{Definition}
 A \textit{general linear Lie algebra} $\mathfrak{gl}(n, F)$ is $\M(n, F)$ with Lie multiplication $[\;,]$ defined as above.

\end{Definition}

In fact, we can make any associative algebra $A$ over $F$ into a Lie algebra with a \textit{commutator} defined in \ref{comt}.

\begin{Notation}
An \textit{$F$-algebra}\index{F-algebra} is an algebra over $F$.
\end{Notation}

\begin{Definition}\label{comt}
Given any associative $F$-algebra $A$, a \textit{commutatior}\index{commutator} is the multiplication $[\;,]$ defined by $[X,Y]=XY-YX$ for all $X, Y \in A$.

\end{Definition}

\begin{Corollary}
A commutator of any given associative $F$-algebra is Lie multiplication.

\end{Corollary}

\begin{proof}
Although the commutator is the generalised definition of Lie multiplication defined in \ref{matlie}, the proof is exactly the same as that given in Example~\ref{matlie}.
\end{proof}

\begin{Definition}\label{algLiealg}

For any associative $F$-algebra $A$, $A_L$ is a Lie algebra where the Lie multiplication of $A_L$ is the commutator.

\end{Definition}

\begin{Lemma}

Let $L$ be a Lie algebra with Lie multiplication $f$. Then $f$ is \textit{skew-symmetry} i.e., $f(y, z)=-f(z, y)$ for all $y, z \in L$.

\end{Lemma}

\begin{proof}
By setting $x=y+z$,
$$f(x, x)= f(y+z, y+z)=f(y,y)+f(y,z)+f(z,y)+f(z,z).$$
By \ref{L1}, $f(x, x)=f(y, y)=f(z, z)=0$. This implies $0=f(z,y)+f(y, z)$ and the result follows.
\end{proof}

\begin{Notation}
Given a Lie algebra $L$ with its commutator, \ref{L1} and \ref{Jacid} can be written as
\begin{equation} \label{L1'}
[x, x]=0 \;\;\; \text{for all } x\in L
\end{equation}
\begin{equation} \label{L2'}
[[x,y],z]+[[y,z],x]+[[z,x],y] = 0 \;\;\; \text{for all} \; x, y, z \in L.
\end{equation}

\end{Notation}

Note that ideals\index{ideals} can be defined in any Lie algebra framework since Lie algebra is $F$-algebra.

\section{Representations of Lie Algebras}

Given a Lie algebra $L$, it is important to closely look at its representations.

\begin{Definition}

For $F$-algebra $V$ and $U$, \textit{$F$-algebra homomorphism} or simply \textit{$F$-homomorphism} is a map $\phi: V \rightarrow U$ that preserves vector space structures (i.e., addition and scalar multiplication) and multiplication.

\end{Definition}

\begin{Remark}
Given Lie algebras $A$ and $B$, any $k$-algebra homomorphism from $V$ into $U$ preserves Lie multiplication.
\end{Remark}

\begin{Definition}

Let $V$ be a $F$-algebra. An \textit{endomorphism}\index{endomorphism} is $F$-algebra homomorphism from $V$ into itself. \end{Definition}

\begin{Lemma}
The set of $F$-endomorphisms of $V$ forms $F$-algebra with multiplication as map composition.
\end{Lemma}

\begin{proof}
See \cite[Section 7]{EW}
\end{proof}

\begin{Definition}
Let $V$ be $F$-algebra. The algebra of $F$-endomorphisms of $V$ is denoted by $\End_F(V)$ or simply $\End(V)$.

\end{Definition}
Recall that by \ref{algLiealg}, given any associative \( F \)-algebra \( A \), a Lie algebra \( A_L \) is defined with Lie multiplication \([X, Y] = XY - YX\) for all \( X, Y \in A \). Since \( \End(V) \) is an associative \( F \)-algebra, it gives rise to a Lie algebra \( \End(V)_L \).
\begin{Definition}
Let \( V \) be an \( F \)-vector space. A \textit{general linear Lie algebra}\index{general linear Lie algebra} \( \mathfrak{gl}(V) \) is the Lie algebra \( \End(V)_L \). The corresponding Lie multiplication in \( \mathfrak{gl}(V) \) is given by \([f, g] = f \circ g - g \circ f\) for all \( f, g \in \mathfrak{gl}(V) \), where \( \circ \) denotes composition of maps. Subalgebras of \( \mathfrak{gl}(V) \) are called \textit{linear Lie algebras}\index{linear Lie algebras}.
\end{Definition}

\begin{Remark}

\item[(i)]\label{End(iv)}
From linear algebra, when \( V \) is finite-dimensional over \( F \), the Lie algebra \( \End(V)_L \) is isomorphic to \( \mathfrak{gl}(n, F) \), where \( \dim V = n \).
\item[(ii)]
Any Lie algebra \( L \) is an \( F \)-vector space. Thus, \( \mathfrak{gl}(V) \), where \( V \) is a Lie algebra \( L \), is \( \mathfrak{gl}(L) = \End(L)_L \).
\end{Remark}

\begin{Definition}
Let \( L \) be a Lie algebra and \( V \) a \( F \)-vector space. A \textit{representation}\index{representation} of \( L \) in \( V \) is an \( F \)-homomorphism \( \rho: L \rightarrow \mathfrak{gl}(V) \).

\end{Definition}

\begin{Remark}

Alternatively, $\rho: L \rightarrow \gl(V)$ is a \textit{representation} of $L$
if $\rho([X,Y])=[\rho(X),\rho(Y)]$.

\end{Remark}

\begin{Remark}
If \( V \) is a finite-dimensional vector space, then \( \gl(V) \) is isomorphic to \( \gl(n, F) \), where \( n = \dim V \). Hence, there exists a map \( \rho' \colon L \to \gl(n, F) \) such that \( \rho'([x, y]) = [\rho'(x), \rho'(y)] \). Sometimes, \( \rho' \) is alternatively considered a representation when \( V \) is finite-dimensional.

\end{Remark}

An important example of a representation of \( L \) is the \textit{adjoint representation}.

\begin{Definition}\label{adDef}

Let $L$ be a Lie algebra.

\item[(i)] Let $x\in L$. The $F$-homomorphism $\ad\;x: L \rightarrow L$ defined by $(\ad x)(a)=[x,a]$ is called the \textit{adjoint representation}\index{adjoint representation} of $x$.

\item[(ii)] The mapping $\ad: L \rightarrow \End(L)$ defined by $\ad(x)= \ad x$ is called the \textit{adjoint representation}\index{adjoint representation} of $L$.

\end{Definition}

\begin{Remark}
The Jacobi identity in \ref{L2'} can be expressed using the adjoint representation \( \ad \). The Jacobi identity is given by
\[
(\ad z \circ \ad x)(y) + (\ad x \circ \ad y)(z) + (\ad y \circ \ad z)(x) = 0.
\]
\end{Remark}

\begin{Definition}

The Lie algebra $L$ is \textit{nilpotent}\index{nilpotent} if there exists a positive integer $m$ such that $L^m=0$. i.e., $l_1l_2\cdots l_m=0$ for any $l_1,l_2,\cdots, l_m \in L$.

\end{Definition}

\section{Solvable and Semisimple Lie Algebras}

\begin{Definition}

A non-abelian Lie algebra $L$ is said to be \textit{simple}\index{simple} if $L$ does not contain non-trivial ideals. i.e., $\{0\}$ and $L$ are the only ideals of $L$.

\end{Definition}

In Lie theory, solvable Lie algebras are defined alongside derived Lie algebras. The definition of a semi-simple Lie algebra follows after the definition of solvable Lie algebras.

\begin{Remark}
Let $L$ be a Lie algebra. The set $L^{(2)}=[L,L]:=\; < [l_1,l_2] \in L \;|\; l_1, l_2 \in L >_F$ forms Lie algebra. $L^{(2)}$ is Lie subalgebra of $L$.
\end{Remark}

\begin{Definition}
Let $L$ be a Lie algebra.
\item[1.] $L^{(2)}$ is called the \textit{derived}\index{derived} Lie subalgebra\index{derived Lie algebra} of $L$.
\item[2.] The sequence of Lie algebras $(L^{(n)})_{n\in \mathbb{N}}$ where $L^{(1)} := L$ and $L^{(n+1)} := [L^{(n)},L^{(n)}]$ for any $n\in \mathbb{N}$ is called the \textit{derived series}\index{derived series} of $L$.
\item[3.] $L$ is \textit{solvable}\index{solvable} if there exists $n\in \mathbb{N}$ such that $L^{(n)}=0$.

In other words, a Lie algebra \( L \) is solvable if its derived series eventually becomes stationary and terminates at \( 0 \).

\end{Definition}

In Lie theory, for any finite-dimensional Lie algebra \( L \), there exists a unique solvable ideal of \( L \) that contains every solvable ideal of \( L \).

\begin{Lemma}\label{solvablelem}

\item[(i)] Subalgebras and homomorphic images of solvable Lie algebras $L$ are solvable.

\item[(ii)] If a Lie algebra $L$ has an ideal $I$ such that $I$ and $L/I$ are solvable, then $L$ is also solvable.

\item[(iii)] The sum $I+J$ of two solvable ideals $I, J$ of a Lie algebra $L$ is solvable ideal of $L$.

\end{Lemma}

\begin{proof}

For the proof of this, see \cite[Lemma 1.7.2]{SF}.
\end{proof}

\begin{Corollary}

Let $L$ be a finite-dimensional Lie algebra. Then there exists a unique solvable ideal of $L$ containing every solvable ideal of $L$. 
\end{Corollary}

\begin{proof}
Let \( R \) be a solvable ideal of largest possible dimension. Assume that \( I \) is any solvable ideal of \( L \). Then, by Lemma \ref{solvablelem} (iii), \( R + I \) is a solvable ideal. Note that \( R \subset R + I \) implies \( \dim R \leq \dim(R + I) \). By the maximality of \( R \) in dimension, \( \dim R = \dim(R + I) \), and hence \( R = (R + I) \). Therefore, \( I \) is contained in \( R \).
\end{proof}

The above proof refers to \cite[Corollary 4.5]{EW}

\begin{Definition}
The maximal solvable ideal of \( L \) is called the \textit{radical}\index{radical} of \( L \) and denoted by \( \rad(L) \).
\end{Definition}

\begin{Definition}
A Lie algebra $L$ is \textit{seimisimple}\index{semisimple} if $\rad(L)=0$.

\end{Definition}

\section{The Killing Form}

\begin{Definition}
Let $L$ be a Lie algebra. The \textit{killing form}\index{killing form} of $L$ is a symmetric bilinear form on $L$ defined by $\kappa(x, y) = \tr(\ad x \circ \; \ad y)$ for all $x, y\in L$. 
\end{Definition}

\begin{Lemma}
$\kappa$ is associative, i.e., $\kappa([x, y], z)=\kappa(x, [y,z])$ for all $x, y, z \in L$. 
\end{Lemma}

\begin{proof}
For any $x, y, z \in L$,
$$\kappa([x,y], z)=\tr(\ad \; [x,y]\; \circ\; \ad z)$$
$$\qquad \qquad \qquad \qquad \qquad \qquad \qquad \qquad\; =\tr((\ad x\; \circ\; \ad y) \; \circ\; \ad z)-\tr((\ad y\; \circ\; \ad x)\; \circ\; \ad z)$$
$$\qquad \qquad \qquad \qquad \qquad \qquad \qquad \qquad\;=\tr(\ad x\; \circ\; (\ad y\; \circ\; \ad z))-\tr(\ad x\; \circ\; (\ad z\; \circ\; \ad y))$$
$$\qquad \quad \;\;\; =\tr(\ad x, \ad \; [y,z])$$
$$\quad =\kappa(x,[y,z])$$.
\end{proof}

The lemma refers to \cite[Definition 2.1]{LN7}.

\begin{Definition}

The Killing form $\kappa$ of a Lie algebra $L$ is said to be \textit{non-degenerate}\index{non-degenerate} if 
\[
\operatorname{rad}(\kappa) := \{x \in L \mid \kappa(x, y) = 0 \text{ for all } y \in L\} = \{0\}.
\]

\end{Definition}

\chapter{Base fields for Lie Algebras}

\section{Algebraically Closed Fields}\index{algebraically closed field}

If a field \( F \) is algebraically closed, then it is possible to solve any polynomial equation over \( F \). In other words, any polynomial over \( F \) can be factorized into linear polynomials over \( F \). For any field \( K \), there exists an algebraic closure \(\overline{K}\). Note that \(\overline{K}\) has the same characteristic as \( K \). For example, \(\overline{\mathbb{F}_p}\) is an algebraic closure of \(\mathbb{F}_p\), and both have characteristic \( p \).

\subsection{Algebraically Closed Fields in Lie Theory}\label{ACFLT}

In Lie theory, many theorems assume that the base field \( F \) is algebraically closed in order to guarantee the existence of eigenvalues\index{eigenvalues} for the map \( \text{ad}\,x \) for any \( x \in L \) in \( F \). Without assuming that \( F \) is algebraically closed, we cannot necessarily conclude that \( F \) contains the eigenvalues of \( \text{ad}\,x \); however, this holds true when the field is algebraically closed. Therefore, all theorems in Chapter~\ref{Chap3} require \( F \) to be algebraically closed to ensure that \( F \) contains all the necessary eigenvalues.

In linear algebra, it is known that:

$f \in \End(V)$ can be represented by a triangular matrix \textit{if and only if} all eigenvalues of $f$ are contained in $F$.

This can be simply explained as follows. Every upper triangular matrix has eigenvalues on diagonal. So any eigenvalue of the upper triangular matrix should be in $F$. On the other hand, if all eigenvalues of $f \in \End(V)$ are contained in $F$, then there exists a change of basis matrix $P$ that converts the corresponding matrix of $f$ into an upper triangular matrix.

\( f \in \End(V) \) is triangularizable if \( F \) is algebraically closed; otherwise, some elements in a solvable Lie algebra may not be representable as upper triangular matrices.

Using this insight, Lie's theorem over the complex field can be interpreted as follows:

\begin{Example}
Let \( L \subset \gl(V) \) be a solvable Lie algebra over \( \mathbb{C} \). Then Lie's theorem implies that there is a basis of \( V \) such that all elements of \( L \) are represented by upper triangular matrices.
\end{Example}

In Section~\ref{Lie's}, we will discuss the theorem more deeply in terms of the base field. For example, the assumption of the complex field \( \mathbb{C} \) in the above example is replaced with that of an algebraically closed field.

\chapter{Special Linear Lie Algebras and $\fsl(2, F)$}

Special linear Lie algebras \( \mathfrak{sl}(2, F) \) are well-studied in Lie theory. They are crucial examples; in particular, \( \mathfrak{sl}(2, \mathbb{C}) \) is an especially important topic. This is because special linear Lie algebras are among the classified classes of Lie algebras. In this chapter, the definition of special linear Lie algebras will be provided, and a new linear Lie subalgebra \( \fsl(2,F) \) over a field \( F \) of characteristic \( 2 \) will be introduced.

\section{Special Linear Lie Algebras}\label{slla}
In this section, $\mathfrak{sl}(2, F)$ over arbitrary field $F$ will be defined.

\subsection{Special Linear Lie Algebras and $\mathfrak{sl}(2,F)$}

\begin{Definition}

The linear Lie subalgebra $\mathfrak{sl}(n, F)$ of $\gl(n, F)$ consisting of all matrices of trace $0$ is called \textit{special linear Lie algebra}\index{special linear Lie algebra}.

\end{Definition}

\begin{Lemma}
$\mathfrak{sl}(n,F)$ is a Lie subalgebra of $\gl(n,F)$.
\end{Lemma}

\begin{proof}
It is easy to check that $\mathfrak{sl}(n,F)$ is closed as a vector subspace.
Hence, it is sufficient to show $\mathfrak{sl}(n,F)$ is closed under Lie multiplication. Since $\tr(AB) = \tr(BA)$ for any $A, B \in \M(n, F)$,
$$\tr([X,Y]) = \tr(XY - YX) = \tr(XY) - \tr(YX) = 0 \quad \forall X, Y \in \mathfrak{sl}(n,F).$$
This implies $[X,Y] \in \mathfrak{sl}(n,F)$.
\end{proof}

\begin{Remark}\label{calsl}
Let $F$ be a field. Then $\mathfrak{sl}(2, F)$ has basis $\{E = e_{12}, F = e_{21}, H = e_{11} - e_{22}\}$. Note that 
$$[E,F] = H, \quad [H,E] = 2E, \quad [H,F] = -2F.$$

This implies that $\mathfrak{sl}(2, F)$ is generated by $E$, $F$, and $H$, and is closed under Lie multiplication, so $\mathfrak{sl}(2, F)$ is a Lie subalgebra of $\gl(2, F)$.

\end{Remark}

\begin{Lemma}

$\mathfrak{sl}(2, \mathbb{C})$ is semisimple.

\end{Lemma}

\begin{proof}

Suppose that the Lie algebra \( \mathfrak{sl}(2, \mathbb{C}) \) has a non-trivial ideal \( I \), that is, \( I \neq \{0\} \) and \( I \neq \mathfrak{sl}(2, \mathbb{C}) \). Let \( e \in I \). Then 
$$h = [e, f] = ef - fe \in I,$$
 and thus
 $$ -2f = [h, f] = hf - fh \in I. $$ This implies that \( I \) contains all generators, so \( I = \mathfrak{sl}(2, \mathbb{C}) \), which is a contradiction. Similarly, if \( f \in I \), then $$h = [e, f] \in I, $$ and thus $$ 2e = [h, e] = he - eh \in I. $$ Similar to the case when \( e \in I \), this leads to a contradiction. Finally, if \( h \in I \), then $$ 2e = [h, e] \quad\text{ and }  -2f = [h, f] $$ are also in \( I \), again leading to a contradiction. Therefore, since \( I \) must contain one of \( e \), \( f \), or \( h \) and any such inclusion leads to \( I = \mathfrak{sl}(2, \mathbb{C}) \), no non-trivial proper ideal exists in \( \mathfrak{sl}(2, \mathbb{C}) \).
\end{proof}

\section{$\fsl(2,F)$}

In the previous section \ref{slla}, the base field in \( \mathfrak{sl}(2, F) \) was assumed to be any field. However, if \( p = 2 \), then the Lie brackets \( [H, E] \) and \( [H, F] \) both become zero. This raises the question of how the structure of \( \mathfrak{sl}(2, F) \) changes when the characteristic of \( F \) is exactly \( 2 \).

In this section, we introduce a new simple Lie algebra, denoted by \( \fsl(2, F) \), of dimension 3, specifically defined for fields \( F \) of characteristic \( 2 \).

Assume throughout this section that the base field \( F \) has characteristic \( p = 2 \).

\subsection{Definition and Properties of $\fsl(2,F)$}\label{defpropfsl}

\begin{Definition}
The linear Lie subalgebra $\fsl(2, F)$ or $\fsl_2$ of $\gl(2, F)$ generated by $e$, $f$, and $h$ satisfying 
$$[h,e] = e, \quad [h,f] = -f, \quad [e,f] = h$$
is called the \textit{fake special linear Lie algebra}\index{fake special linear Lie algebra}.

\end{Definition}

\begin{Lemma}
$\fsl(2,F)$ is a simple Lie algebra.

\end{Lemma}

\begin{proof}
Let assume that the Lie algebra has a non-trivial ideal $I$, i.e., $I\neq \{0\}$ and $\fsl(2, F)$. Let $e\in I$. Then $$[e,f]=h \in I \quad \text{ and thus } \quad [h,f]=-f \in I.$$ This implies that $I$ contains all generators, so $I=\fsl(2, F)$, which is contraction. Likewise, if $f \in I$, then $$[e,f]=h \in I \quad \text{and thus} \quad [h,e]=e \in I.$$ Hence $I=\fsl(2, F)$, so this is not the case. Lastly, if $h\in I$, then $$[h,e]=e \quad\text{and}\quad [h,f]=f \in I$$ so $I=\fsl(2, F)$. This also leads to contradiction. Since $I$ contains any of $e, f$ and $h$, non-trivial ideal $I$ does not exist in $\fsl(2, F)$.
\end{proof}

\begin{Lemma}

$\fsl(2,F)$ is not solvable.

\end{Lemma}

\begin{proof}

By the definition of $\fsl(2,F)$, its derived Lie subalgebra 
$$\fsl(2,F)^{(2)} = [\fsl(2,F), \fsl(2,F)]$$ 
has generators $e$, $f$, and $h$. In other words, 
$$[\fsl(2,F), \fsl(2,F)] = \fsl(2,F).$$ 
Hence, 
$$\fsl(2,F)^{(n)} = \fsl(2,F) \quad \text{for all } n \in \mathbb{N}.$$
Since the derived series $(\fsl(2,F)^{(n)})_{n \in \mathbb{N}}$ does not terminate to $0$, $\fsl(2,F)$ is not solvable.
\end{proof}

From this discussion, $\fsl(2,F)$ is simple but not solvable.

\chapter{Lie Theory over Field of Prime Characteristic}\label{Chap3}\index{Lie theory}

Lie algebras over different base fields exhibit different properties, resulting in theorems that may hold in some fields but not in others. As shown in \ref{Lie's}, \ref{Cart}, \ref{Jord}, \ref{RepLie}, and \ref{Weyl}, all of the following theorems require the assumption that the field \( F \) is algebraically closed or has characteristic zero. It turns out that these theorems do not hold when the base field has characteristic \( p > 0 \). In other words, counterexamples arise when the base field has prime characteristic.

Since this project focuses on understanding why Lie theory theorems may fail over fields of prime characteristic, and on presenting relevant counterexamples, we do not provide full proofs of the classical theorems. These proofs are readily available in \cite{EW}. Furthermore, when a Lie theory theorem is referenced in this project, its proof is typically omitted and replaced by a citation to \cite{EW}, as detailed proofs are not the main focus here.

\section{Lie's Theorem over Field of Prime Characteristic}\label{Lie's}

\begin{Theorem}[Lie's Theorem]\index{Lie's Theorem}

Let $F$ be algebraically closed, $V$ an $n$-dimensional $F$-vector space, and $L \subset \gl(V)$ a solvable linear Lie algebra. Then there exists an associative isomorphism $f: \End(V) \rightarrow \gl(n, F)$ such that $f$ maps every element in $L$ onto an upper triangular matrix.

\end{Theorem}
\begin{proof}
See \cite[Section 9.2]{FH}.
\end{proof}

An important assumption of Lie's theorem is that $F$ is algebraically closed. Otherwise, there is a counterexample, which is the following.

\begin{Example}

Consider $F = \mathbb{R}$ and a linear Lie subalgebra 
$$ L = \left\langle \begin{pmatrix} 
0 & -1 \\ 
1 & 0 
\end{pmatrix} \right\rangle \subset \gl(2, \mathbb{R}). $$
Note that the characteristic polynomial of the generator of $L$ is $\lambda^2 + 1 = 0$, so $\lambda = i, -i \notin \mathbb{R}$. This implies that the generator cannot be represented as an upper triangular matrix, and thus neither can any element in $L$. See \ref{ACFLT}. 

Indeed, if the generator were represented as an upper triangular matrix, its characteristic polynomial would be $(\lambda - i)(\lambda + i) = 0$. However, we cannot write $i, -i$ as diagonal entries because the matrix would not be over $F = \mathbb{R}$. This shows that Lie's theorem does not hold when $F$ is not algebraically closed.

\end{Example}

\subsection{A Counterexample of Lie's Theorem}

A counterexample exists that causes Lie's theorem to not hold if $F$ has prime characteristic. The following example refers to \cite[Exercise 3]{H} and \cite[Exercise 6.4]{EW}.

\begin{Example}

Let $p$ be prime and let $F$ be a field of characteristic $p$. Consider the $p \times p$ matrices $x$ and $y$:

$$
x = \begin{pmatrix}
0 & 1 & 0 & \cdots & 0 \\
0 & 0 & 1 & \cdots & 0 \\
\vdots & \vdots & \vdots & \ddots & \vdots \\
0 & 0 & 0 & \cdots & 1 \\
1 & 0 & 0 & \cdots & 0
\end{pmatrix}, \qquad
y = \begin{pmatrix}
0 & 0 & \cdots & 0 & 0 \\
0 & 1 & \cdots & 0 & 0 \\
\vdots & \vdots & \ddots & \vdots & \vdots \\
0 & 0 & \cdots & p-2 & 0 \\
0 & 0 & \cdots & 0 & p-1
\end{pmatrix}.
$$

Let $S$ be a $2$-dimensional subalgebra of $\gl(n, F)$ generated by $x$ and $y$. Recall that $S^{(2)} = [S, S] = \langle [s, t] \;|\; s, t \in S \rangle_F$ is the derived Lie subalgebra of $S$. Consider the commutator of two basis elements $x$ and $y$ in $S$ in order to find generators of $S^{(2)}$.

\[
[x, y] = 
\begin{pmatrix}
0 & 1 & 0 & \cdots & 0 \\
0 & 0 & 2 & \cdots & 0 \\
\vdots & \vdots & \vdots & \ddots & \vdots \\
0 & 0 & 0 & \cdots & p-1 \\
0 & 0 & 0 & \cdots & 0
\end{pmatrix}
- 
\begin{pmatrix}
0 & 0 & 0 & \cdots & 0 & 0 \\
0 & 0 & 1 & \cdots & 0 & 0 \\
\vdots & \vdots & \vdots & \ddots & \vdots & \vdots \\
0 & 0 & 0 & \cdots & 0 & p-2 \\
p-1 & 0 & 0 & \cdots & 0 & 0
\end{pmatrix}
\]

\[
= 
\begin{pmatrix}
0 & 1 & 0 & \cdots & 0 \\
0 & 0 & 1 & \cdots & 0 \\
\vdots & \vdots & \vdots & \ddots & \vdots \\
0 & 0 & 0 & \cdots & 1 \\
-(p-1) & 0 & 0 & \cdots & 0
\end{pmatrix}
= 
\begin{pmatrix}
0 & 1 & 0 & \cdots & 0 \\
0 & 0 & 1 & \cdots & 0 \\
\vdots & \vdots & \vdots & \ddots & \vdots \\
0 & 0 & 0 & \cdots & 1 \\
1 & 0 & 0 & \cdots & 0
\end{pmatrix}
= x.
\]

Hence, $S^{(2)}$ is a subalgebra of $S$ generated by $x$. Then $S^{(3)} = [S^{(2)}, S^{(2)}]$ is zero since $S$ is anti-commutative, so $[x, x] = 0$ and $x$ is the only generator of $S^{(2)}$. This implies that $S$ is solvable.

To find eigenvectors $v$ of $x$, the equation $(x - \lambda I)v = 0$ shows that

$$
\begin{pmatrix}
-\lambda & 1 & 0 & \cdots & 0 \\
0 & -\lambda & 1 & \cdots & 0 \\
\vdots & \vdots & \vdots & \ddots & \vdots \\
0 & 0 & 0 & \cdots & 1 \\
1 & 0 & 0 & \cdots & -\lambda
\end{pmatrix}
\begin{pmatrix}
v_1 \\
v_2 \\
\vdots \\
v_{p-1} \\
v_p
\end{pmatrix}
= 0
$$

$$
\begin{pmatrix}
-\lambda v_1 + v_2 \\
-\lambda v_2 + v_3 \\
\vdots \\
-\lambda v_{p-1} + v_p \\
-\lambda v_p + v_1
\end{pmatrix}
= 0.
$$

Every entry should be zero, so $v_{k+1} = \lambda v_k$ for all $k = 1, 2, \cdots, p-1$ and $v_1 = \lambda v_p$. This means $v_{k+1} = \lambda^k v_1$ for all $k = 1, 2, \cdots, p-1$ and $v_1 = \lambda^p v_1$. Hence, $(\lambda^p - 1)v_1 = 0$. If $v_1 = 0$, then $v_1 = v_2 = \cdots = v_p = 0$, so no eigenvectors in this case. Thus, $\lambda^p = 1$. Since $(\lambda - 1)^p = \lambda^p + \binom{p}{1} \lambda^{p-1} + \binom{p}{2} \lambda^{p-2} + \cdots + (-1)^p = \lambda^p - 1 = 0$, this implies that $\lambda = 1$. Thus, the eigenspace $E_1$ is generated by $e_1 + e_2 + \cdots + e_p$, i.e., $E_1 = \langle e_1 + \cdots + e_p \rangle$.

For $(y - \lambda I)u = 0$,

$$
\begin{pmatrix}
-\lambda & 0 & \cdots & 0 & 0 \\
0 & 1 - \lambda & \cdots & 0 & 0 \\
\vdots & \vdots & \ddots & \vdots & \vdots \\
0 & 0 & \cdots & p - 2 - \lambda & 0 \\
0 & 0 & \cdots & 0 & p - 1 - \lambda
\end{pmatrix}
\begin{pmatrix}
u_1 \\
u_2 \\
\vdots \\
u_{p-1} \\
u_p
\end{pmatrix}
= 0
$$

$$
\begin{pmatrix}
-\lambda u_1 \\
(1 - \lambda) u_2 \\
\vdots \\
(p - 2 - \lambda) u_{p-1} \\
(p - 1 - \lambda) u_p
\end{pmatrix}
= 0.
$$

When $\lambda = i$ where $i = 0, 1, \cdots, p - 1$, the corresponding eigenspace $E_i'$ for $y$ is generated by $e_{i+1}$, i.e., $E_0' = \langle e_1 \rangle, E_1' = \langle e_2 \rangle, \cdots, E_{p-1}' = \langle e_p \rangle$.

Since each eigenspace of $y$ does not intersect with the eigenspace $E_1$ of $x$, $x$ and $y$ have no common eigenvector. Thus, it is not true that every element in $S$ can be mapped onto an upper triangular matrix.
\end{Example}

\section{Cartan's Criterion over Field of Prime Characteristic}\label{Cart}

As Lie's theorem does not hold in prime characteristic fields, another well-known theorem called \textit{Cartan's Criterion} is needed to be checked whether this theorem holds in prime characteristic fields or not. The answer is no, since this theorem has a counterexample $\fsl(2,\mathbb{C})$ over a field with characteristic $2$.

\begin{Theorem}[Cartan's Criterion]\label{cc}\index{Cartan's Criterion}

Let $L$ be a Lie algebra over $F$ of characteristic $0$ and $\rho$ be a representation of $L$. Then the following statements are equivalent:

\item[1.]
$\tr(\rho(a) \circ \rho(b))=0 \quad \text{for} \; a\in L, b \in [L,L]$.

\item[2.]
$\tr(\rho(a) \circ \rho(a))=0 \quad \text{for all} \; a\in[L,L]$.

\item[3.]
$L$ is solvable.

\end{Theorem}

\begin{proof}
See \cite[Chapter 1]{SF}.
\end{proof}

\subsection{A Counterexample of Cartan's Criterion: $\fsl(2,F)$}

$\mathfrak{fsl}(2, F)$ is a counterexample to Cartan's Criterion when the base field has characteristic 2.

\begin{Example}

Consider the adjoint representation $\ad$ of $L$ for Cartan's Criterion. Then the second equivalent statement is

\begin{equation}\label{sec}
\tr(\ad a \circ \ad a)=0 \quad \text{for all} \; a \in [\mathfrak{sl}(2,F), \mathfrak{sl}(2,F)].
\end{equation}

Since $[\mathfrak{sl}(2,F), \mathfrak{sl}(2,F)]$ has the same generators as $\mathfrak{sl}(2,F)$, it is sufficient to check the traces of $\ad e \circ \ad e$, $\ad f \circ \ad f$, and $\ad h \circ \ad h$ with respect to the basis $\{e, f, h\}$.

\[
(\ad e \circ \ad e)(e) = 0, \quad (\ad e \circ \ad e)(f) = -e, \quad (\ad e \circ \ad e)(h) = 0.
\]
Therefore, the corresponding matrix of $(\ad e \circ \ad e)$ is
\[
\begin{pmatrix}
0 & -1 & 0 \\
0 & 0 & 0 \\
0 & 0 & 0
\end{pmatrix}.
\]
\[
(\ad f \circ \ad f)(e) = f, \quad (\ad f \circ \ad f)(f) = 0, \quad (\ad f \circ \ad f)(h) = 0,
\]
and the corresponding matrix of $(\ad f \circ \ad f)$ is
\[
\begin{pmatrix}
0 & 0 & 0 \\
1 & 0 & 0 \\
0 & 0 & 0
\end{pmatrix}.
\]
\[
(\ad h \circ \ad h)(e) = e, \quad (\ad h \circ \ad h)(f) = f, \quad (\ad h \circ \ad h)(h) = 0.
\]
The corresponding matrix of $(\ad h \circ \ad h)$ is
\[
\begin{pmatrix}
1 & 0 & 0 \\
0 & 1 & 0 \\
0 & 0 & 0
\end{pmatrix}.
\]

Hence, the trace of each is $\tr(\ad e \circ \ad e) = 0$, $\tr(\ad f \circ \ad f) = 0$, and $\tr(\ad h \circ \ad h) = 2 = 0$ since $\operatorname{char} F = 2$. This implies that $\mathfrak{sl}(2, F)$ satisfies the second statement \eqref{sec} of Cartan's criterion.

However, by the previous discussion in \ref{defpropfsl}, $\mathfrak{sl}(2, F)$ is not solvable. Therefore, there exists a counterexample $\mathfrak{sl}(2, F)$ of Cartan's Criterion in a base field with characteristic $2$. This implies that Cartan's Criterion is not satisfied in Lie algebras of prime characteristic.

\end{Example}

\subsection{Cartan's Criterion and the Killing Form over Field of Prime Characteristic}

By the definition of the Killing form, the corollary follows.

\begin{Corollary}\label{kcc}

A finite-dimensional Lie algebra $L$ over $F$ with characteristic $0$ is solvable if and only if $\kappa(L, [L, L])=0$.

\end{Corollary}

\begin{proof}

Put \( \rho = \ad \) in Cartan's Criterion \ref{cc}. Then the first equivalent statement can be expressed with the Killing form as
$$ \kappa(a, b) = 0 \quad \text{for all } a \in L, b \in [L, L]. $$
Since this statement is equivalent to \( \kappa(L, [L, L]) = 0 \), the result follows.
\end{proof}

From \cite[Theorem 3.7]{IS}, Cartan's criterion can be stated as the following.

\begin{Theorem}[Cartan's Criterion for Semisimplicity]\label{CCS}

Let the base field be $\mathbb{C}$.  
A Lie algebra $L$ is semisimple if and only if its Killing form $\kappa$ is non-degenerate.

\end{Theorem}

\begin{proof}

If $L$ is semisimple, then $\rad(L) = \{0\}$. Since $\rad(\kappa)$ is a solvable ideal by \cite{EW}, it follows that $\rad(\kappa) \subset \rad(L) = \{0\}$. Thus, $\kappa$ is non-degenerate.

For the converse, suppose the Killing form $\kappa$ is non-degenerate, i.e., $\rad(\kappa) = \{0\}$. Let $H \subset L$ be an abelian ideal. Then $H$ is $L$-stable under the adjoint representation $\rho = \ad$.

Define the map $\phi = \ad(x) \circ \ad(y) \in \gl(L)$ for $(x, y) \in H \times L$. We show that $\phi|_H = 0$ and $\phi|_{L/H} = 0$.

For any $z \in H$, since $[y, z] \in H$ and $H$ is abelian, we have
\[
\phi|_H(z) = [x, [y, z]] = 0.
\]

Now pick $z \in L$. Then in the quotient $L/H$, we have
\[
\phi|_{L/H}(z + H) = [x, [y, z]] + H = 0 + H,
\]
since $x \in H$ and $[x, [y, z]] \in H$.

Therefore,
\[
\kappa(x, y) = \operatorname{tr}(\phi) = \operatorname{tr}(\phi|_H) + \operatorname{tr}(\phi|_{L/H}) = 0
\]
for all $(x, y) \in H \times L$. This implies $H \subset \rad(\kappa) = \{0\}$, and hence $H = \{0\}$. Thus, $L$ has no nonzero abelian ideals and is therefore semisimple by \cite{EW}.
\end{proof}

In the previous section, we showed that $\fsl(2,F)$ is simple but not semisimple. Moreover, $\fsl(2,F)$ is not semisimple, but its Killing form is non-degenerate. Hence, $\fsl(2,F)$ is a counterexample to \ref{CCS} in a prime characteristic field.

\section{Preservation of Jordan Decomposition over Field of Prime Characteristic}\label{Jord}

\begin{Definition}

The \textit{Jordan decomposition}\index{Jordan decomposition} of $x$ is the unique expression of $x$ as $x=d+n$ where $d ,n \in \End (V)$ are such that $d$ is diagonalisable, $n$ is nilpotent, and $[d,n]=0$.

\end{Definition}

\begin{Theorem}
Let $L$ be a Lie algebra over algebraically closed field $F$. Then given any $x\in L$ has a Jordan decomposition.

\end{Theorem}

\begin{proof}
See \cite{EW}.
\end{proof}

A counterexample exists in a prime characteristic field.

\begin{Example}
Consider $\gl(2,F)$ where $\operatorname{char} F = 2$. Then there exists a matrix that does not have distinct eigenvalues. Therefore, the matrix does not have the Jordan normal form that is a diagonal matrix.
\end{Example}

However, if the condition of diagonalisable \( d \) changes to semisimple \( s \), Jordan decomposition is preserved for the general linear Lie algebra over a field of prime characteristic.

\begin{Remark}
In characteristic $p$, any matrix $A$ is decomposed into a semisimple matrix $S$ and a nilpotent matrix $N$ such that $[S,N] = 0$.

\end{Remark}

In other words, if we change the condition of \( d \) from being diagonalisable to being semisimple, the Jordan decomposition of the Lie algebra is preserved. We gain the insight that the notion of semisimple is much weaker than diagonalisable and works comparatively well in a prime characteristic field.

\section{Irreducible Representations of Lie Algebras over Field of Prime Characteristic}\label{RepLie}

Representations of semisimple Lie algebras have interesting properties. Note that \( \mathfrak{sl}(2, F) \) is semisimple, particularly \( \mathfrak{sl}(2, \mathbb{C}) \). Irreducible representations of \( \mathfrak{sl}(2, \mathbb{C}) \) are worth discussing, as representations of \( \mathfrak{sl}(2, F) \) where \( F \) has prime characteristic cannot be applied to this discussion.

\subsection{Representations of $\mathfrak{sl}(2, \mathbb{C})$}\label{Repsl2}

This subsection refers to \cite[Chapter 11, Part II]{FH}.

Recall that \( \mathfrak{sl}(2, \mathbb{C}) \) is generated by \( e \), \( f \), and \( h \) satisfying \( [e,f] = h \), \( [h,e] = 2e \), and \( [h,f] = -2f \). Let \( V \) be an irreducible finite-dimensional representation of \( \mathfrak{sl}(2, \mathbb{C}) \). By the preservation of Jordan decomposition \ref{Jord}, the action of \( h \) on \( V \) is diagonalizable. This implies that \( V \) has a decomposition 
$$ V = \bigoplus_{\alpha \in \Lambda(h)} V_{\alpha}, $$
where \( V_{\alpha} \) are irreducible subspaces, and \( \alpha \in \mathbb{C} \) is the eigenvalue of \( h \) for any vector \( v \in V_{\alpha} \). That is,
$$ h(v) = \alpha \cdot v \quad \text{for any} \quad v \in V_{\alpha}. $$

Since we know how \( h \) acts on \( V_{\alpha} \), it is useful to examine how \( e \) and \( f \) act on all of the summands \( V_{\alpha} \). In fact, \( e \) and \( f \) act on \( V_{\alpha} \) in such a way that they carry eigenvectors in \( V_{\alpha} \) to other subspaces \( V_{\alpha'} \). In other words, the image of a given vector \( v \in V_{\alpha} \) under the action of \( e \) lies in another summand of the decomposition.

To see how \( h \) acts on \( e(v) \), we compute
$$ h(e(v)) = e(h(v)) + [h,e](v). $$ 
Substituting \( h(v) = \alpha \cdot v \) and \( [h,e] = 2e \), we get
$$ h(e(v)) = e(\alpha \cdot v) + 2e(v) = (\alpha + 2) \cdot e(v). $$
Thus, if \( v \) is an eigenvector for \( h \) with eigenvalue \( \alpha \), then \( e(v) \) is an eigenvector for \( h \) with eigenvalue \( \alpha + 2 \). Therefore, \( e \) maps \( V_{\alpha} \) to \( V_{\alpha + 2} \), i.e., \( e: V_{\alpha} \to V_{\alpha + 2} \).

Similarly, the action of \( f \) on \( v \in V_{\alpha} \) is computed as
$$ h(f(v)) = f(h(v)) + [h,f](v), $$ 
and since \( [h,f] = -2f \), we have
$$ h(f(v)) = f(\alpha \cdot v) - 2f(v) = (\alpha - 2) \cdot f(v). $$ 
Thus, for any eigenvector \( v \) for \( h \) with eigenvalue \( \alpha \), \( f(v) \) is an eigenvector for \( h \) with eigenvalue \( \alpha - 2 \). In other words, \( f: V_{\alpha} \to V_{\alpha - 2} \).

\begin{figure}[h]
\includegraphics[width=\linewidth]{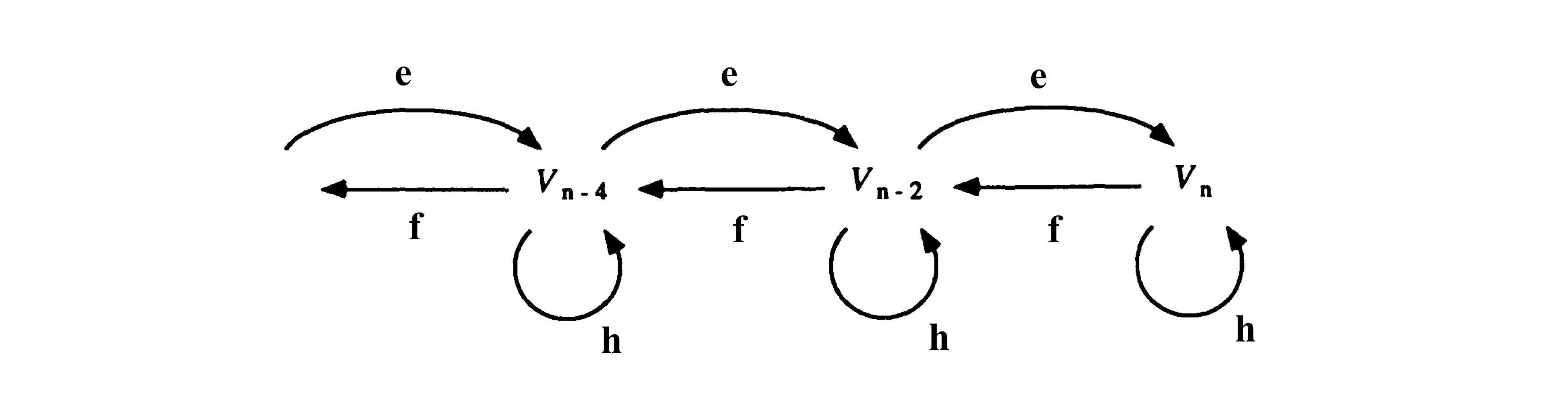}
\centering
\caption{The action of $e$, $h$, and $f$ on the eigenspaces $\{V_{\alpha}\}$.}
\label{figehf}
\end{figure}

From this, and by the irreducibility of \( V \) over \( \mathbb{C} \), all \( V_{\alpha} \) are congruent to each other modulo \( 2 \). Let \( \alpha_0 \) be an actual eigenvalue. Then 
$$ \bigoplus_{n \in \mathbb{Z}} V_{\alpha_0 + 2n} $$
is invariant under \( \mathfrak{sl}(2, \mathbb{C}) \) and is equal to \( V \).

With this observation, there exists a sequence of the form \( \beta, \beta + 2, \cdots, \beta + 2k \), such that
$$ \bigoplus_{0 \leq i \leq k} V_{\beta + 2i} $$
is a subspace of \( V \). Since \( V \) is finite-dimensional, the sequence terminates, so let \( n = \beta + 2k \) be the last element of this sequence.

Choose an eigenvector \( v \in V_n \). Then \( V_{n+2} = (0) \), so \( e(v) = 0 \). This implies that \( e \) cannot carry \( v \) to another eigenspace any further.

Note that the action of \( f \) on \( v \) goes backwards. Consider the set \( \{ v, f(v), f^2(v), \cdots \} \). By the following proof, the set spans \( V \). Let \( W \) be a subspace of \( V \) spanned by the set \( \{ v, f(v), f^2(v), \cdots \} \). Then, \( f \) preserves the subspace \( W \) since it carries the vector \( f^k(v) \) into \( f^{k+1}(v) \). Clearly, since \( v \in V_n \),
$$ f^k(v) \in V_{n-2k}. $$
We also have
$$ h(f^k(v)) = (n - 2k) \cdot f^k(v). $$
This implies that \( h \) also preserves the subspace \( W \). Lastly, we need to check \( e(W) \subset W \). By the previous discussion, \( e(v) = 0 \in W \). Using the definition of the commutator, we have
$$ e(f(v)) = [e,f](v) + f(e(v)) $$
$$ \quad = h(v) + f(0) $$
$$ \quad = n \cdot v + 0 = n \cdot v. $$
Similarly, for \( f^2(v) \),
$$ e(f^2(v)) = [e,f](f(v)) + f(e(f(v))) $$
$$ \quad = h(f(v)) + f(n \cdot v) $$
$$ \quad = (n-2) \cdot f(v) + n \cdot f(v) = (n-2) \cdot f(v) + n \cdot v. $$

Thus, by induction, we have
$$ e(f^k(v)) = \left( n + (n-2) + (n-4) + \cdots + (n-2k+2) \right) \cdot f^{k-1}(v), $$
which simplifies to
$$ e(f^k(v)) = k(n-k+1) \cdot f^{k-1}(v). $$

Therefore, \( e \), \( f \), and \( h \) carry the subspace \( W \) into itself under the action of \( \mathfrak{sl}(2, \mathbb{C}) \), which implies that, by the irreducibility of \( V \), \( \{ v, f(v), f^2(v), \cdots \} \) spans \( V \).

Since \( V \) is finite-dimensional, there exists \( k \) such that \( f^k(v) = 0 \) for sufficiently large \( k \). Let \( m \) be the smallest power of \( f \) that annihilates \( v \). Then, two observations follow. First, the set \( \{ v, f(v), f^2(v), \cdots, f^m(v) \} \) spans \( V \). Second, consider the equation in the above proof:
$$ 0 = e(f^m(v)) = m(n-m+1) \cdot f^{m-1}(v). $$ 
By the minimality of \( m \), \( f^{m-1}(v) \neq 0 \), so \( n - m + 1 = 0 \). This shows that \( n = \beta + 2k \) is a non-negative integer. Combining the two observations, the sequence \( \beta, \beta + 2, \cdots, \beta + 2k \) is a set of integers differing by 2 and symmetric about the origin in \( \mathbb{Z} \).

To summarize, for every non-negative integer \( n \), there is a unique representation \( V^{(n)} \), which is \( (n+1) \)-dimensional, with \( h \) having eigenvalues \( n, n-2, \cdots, -n+2, -n \).

It is desirable to see how some of the standard representations of \( \mathfrak{sl}(2, \mathbb{C}) \) are formed. The trivial one-dimensional representation of \( \mathbb{C} \) is \( V^{(0)} \). Consider the standard representation of \( \mathfrak{sl}(2,\mathbb{C}) \) on \( V = \mathbb{C}^2 \). If \( x \) and \( y \) are the standard basis for \( \mathbb{C}^2 \), then \( h(x) = x \) and \( h(y) = -y \). Therefore, \( V \) can be decomposed into weight spaces, i.e.,
$$ V = \mathbb{C} \cdot x \oplus \mathbb{C} \cdot y = V_{-1} \oplus V_1, $$
which is the representation \( V^{(1)} \) above. Similarly, consider the symmetric square \( W = \Sym^2 V = \Sym^2 \mathbb{C}^2 \). A basis for the symmetric square \( W \) is \( \{ x^2, xy, y^2 \} \), and thus
$$ h(x \cdot x) = x \cdot h(x) + h(x) \cdot x = 2x \cdot x, $$
$$ h(x \cdot y) = x \cdot h(y) + h(x) \cdot y = 0, $$
$$ h(y \cdot y) = y \cdot h(y) + h(y) \cdot y = -2y \cdot y. $$

This shows that the representation
$$ W = \mathbb{C} \cdot x^2 \oplus \mathbb{C} \cdot xy \oplus \mathbb{C} \cdot y^2 = W_{-2} \oplus W_0 \oplus W_2 $$
is the representation \( V^{(2)} \) above. The \( n \)-th symmetric power \( \Sym^n V \) of \( V \), with a basis \( \{ x^n, x^{n-1}y, \cdots, y^n \} \), has the representation \( V^{(n)} \) above. In fact,
$$ h(x^{n-k}y^k) = (n-k) \cdot h(x) \cdot x^{n-k-1} y^k + k \cdot h(y) \cdot x^{n-k} y^{k-1} $$
$$ = (n-2k) \cdot x^{n-k} y^k. $$

Thus, the eigenvalues of \( h \) on \( \Sym^n V \) are exactly \( n, n-2, \cdots, -n \). This implies that \( V^{(n)} = \Sym^n V \). In summary, any irreducible representation of \( \mathfrak{sl}(2, \mathbb{C}) \) is a symmetric power of the standard representation \( V \cong \mathbb{C}^2 \).

Irreducible representations of \( \mathfrak{sl}(2, \mathbb{C}) \) have the above properties, but those of \( \mathfrak{sl}(2, F) \), where \( F \) has characteristic \( p \), do not. For a counterexample, see \ref{sl2weyl}.

\section{Weyl's Theorem on Complete Reducibility over Field of Prime Characteristic}\label{Weyl}\index{Weyl's Theorem on Complete Reducibility}

Weyl's Theorem on complete reducibility is the following:

\begin{Theorem}[Weyl's Theorem]\index{Weyl's Theorem}

Let $L$ be a semisimple Lie algebra over $F$ where its characteristic is zero. Then every finite-dimensional representation of $L$ is completely reducible\index{completely reducible}.

\end{Theorem}

However, Weyl's theorem does not hold in a field of positive characteristic. This phenomenon is related to the fact that in a prime characteristic field, all calculations are reduced modulo \( p \). The following refers to \cite{DM} and \cite{R}.

Before looking at a counterexample of Weyl's theorem in Lie algebras, we can first examine Weyl's theorem in the context of groups.

\begin{Theorem}[Maschke's theorem]\index{Maschke's theorem}\label{MT}
Let $V$ be a representation of a finite group $G$ over a field $F$ with characteristic not dividing $|G|$. If $W$ is a subrepresentation of $V$, then $V=W\oplus U$ for some other subrepresentation $U$.

\end{Theorem}

\begin{Theorem}
If $G$ is a finite group over $\mathbb{C}$, then $G$ acts completely reducibly.
\end{Theorem}
\begin{proof}
This immediately follows from Maschke's theorem \ref{MT}.
\end{proof}

However, Maschke's theorem is not true if the base field $\mathbb{C}$ is replaced by $\overline{\mathbb{F}_p}$, as demonstrated by the following examples.

\begin{Example}
\item[1.]
Let $F=\overline{\mathbb{F}_p}$. Consider the cyclic group $C_p$ of order $p$, where the generator is $\sigma$. Let $V=F^2$ be a vector space over $F$, and suppose that $\sigma$ acts on $V$ by the matrix
$$
\begin{pmatrix}
 1 & 1 \\
 0 & 1
\end{pmatrix}.
$$
Then the representation $\rho: \overline{\mathbb{F}_p} C_p \rightarrow \gl(F^2)$ is reducible.

\item[2.]
Consider the natural action of \( G = SO(V) \) on an odd-dimensional \( F \)-vector space \( V \), where \( \operatorname{char} F = 2 \).
\end{Example}

This shows that Weyl's theorem does not hold over a prime characteristic field in the context of groups. In fact, the theorem does not hold over a base field of prime characteristic in Lie algebras.

\begin{Theorem}
If $L$ is a semisimple Lie algebra over $\mathbb{C}$, then Weyl's theorem implies $L$ is completely reducible.

\end{Theorem}
However, this theorem is not true if $\mathbb{C}$ is replaced by $\overline{\mathbb{F}_p}$, as shown by the counterexample $\mathfrak{sl}(2, \overline{\mathbb{F}_3})$.

Recall that $\mathfrak{sl}(2, F) = \langle e, f, h \mid [e,f] = h, [h,e] = 2e, [h,f] = -2f \rangle$.

\begin{Example}\label{sl2weyl}
 
 Consider $\mathfrak{sl}(2, F)$ as a Lie algebra acting on a submodule $V_2$ defined below, in both cases where $F$ is $\mathbb{C}$ and where $F$ has characteristic $3$. In other words, we want to examine representations of $\mathfrak{sl}(2,F)$ for $F=\mathbb{C}$ and for $F$ with characteristic $3$. Let $V_2$ be a 2-dimensional module generated by $x$ and $y$. Then a basis for $\Sym^3 V_2$ is given by $(x^3, x^2 y, x y^2, y^3)$. Define the action of $e$ as $x \cdot \partial_y$,
$$ x \frac{\partial}{\partial y}, $$
and the action of $f$ as $y \cdot \partial_x$,
$$ y \frac{\partial}{\partial x}, $$
so that the action of $e$ carries each basis element to its left basis element, and the action of $f$ carries each basis element to its right basis element, when we view the basis as an ordered set.

\item[1.]
Over the base field $\mathbb{C}$, $\Sym^3 V_2$ is completely reducible by \ref{Repsl2}.

\item[2.]

Consider $\mathfrak{sl}(2, F)$ where $F$ has characteristic 3. Then the actions of $e$ and $f$ on $x^3$ are both zero, 
$$\text{i.e.,}\quad e(x^3) = 0 \quad \text{and} \quad f(x^3) = 3x^2 y r = 0.$$
Since $h = [e, f]$, we have
$$h(x^3) = 0.$$
Likewise, the actions of $e$, $f$, and $h$ on $y^3$ are all zero. Hence, 
$$U := \langle x^3, y^3 \rangle$$
is a submodule of $\Sym^3 V_2$. However, there is no complement submodule of $U$ in $\Sym^3 V_2$. 

If $W$ is any complement of $U$, then $\Sym^3 V_2 = U \oplus W$, and thus $x^2 y \in W$. Since $W$ is a submodule, it is closed under the action of $\mathfrak{sl}(2, F)$. However,
$$e(x^2 y) = x^3 \in U,$$
which leads to a contradiction.

Therefore, there is no complement of $U$ in $\Sym^3 V_2$, and this implies that $\Sym^3 V_2$ is not completely reducible. Hence, Weyl's theorem does not hold over a field of characteristic 3.

\end{Example}

Weyl's theorem does not guarantee complete reducibility over fields of prime characteristic. To address these limitations, concepts such as Weyl modules and Lusztig's conjecture have been developed to study the structure of representations in this setting.

\subsection{Weyl modules and Lusztig's Conjecture}

For all the details in this section, see \cite[Chapter 3]{HJWMLC}.

Let $G$ be a reductive group.\index{reductive group} For simple $G$-modules with a given highest weight $\lambda \in X^+$, consider the simple module $V(\lambda)_{\mathbb{C}}$ for the corresponding semisimple Lie algebra $\mathfrak{g}_{\mathbb{C}}$ over $\mathbb{C}$. There exists a distinguished $\mathbb{Z}$-form $V(\lambda)_{\mathbb{Z}} \subset V(\lambda)_{\mathbb{C}}$, which allows reduction modulo $p$. Consequently, $G$ acts naturally on the module $F \otimes V(\lambda)_{\mathbb{Z}}$, where $\operatorname{char} F = p$.

\begin{Definition}

A \textit{Weyl module}\index{Weyl module} is a module $V(\lambda)$ defined as $V(\lambda) := F \otimes V(\lambda)_{\mathbb{Z}}$.

\end{Definition}

Unlike a simple module $V(\lambda)_{\mathbb{C}}$ over $\mathbb{C}$, the module $V(\lambda)$ is not guaranteed to be simple. However, for Weyl modules in characteristic $p$, their characters remain unchanged under reduction modulo $p$, and therefore can still be computed using Weyl's character formula\index{Weyl's character formula}. In other words, Weyl's character formula holds for Weyl modules even over fields of prime characteristic.

Define the character of the Weyl module by $\chi(\lambda) := \operatorname{char} V(\lambda)$ and the character of its unique simple quotient $L(\lambda)$ by $\chi_p(\lambda) := \operatorname{char} L(\lambda)$. Then the two characters are related by the following expression:
\[
\chi_p(\lambda) = \sum_{\mu \leq \lambda} \alpha_{\mu} \chi(\mu), \quad \text{with } \alpha_{\mu} \in \mathbb{Z} \text{ and } \alpha_{\lambda} = 1.
\]

The unknown coefficients $\alpha_{\mu}$ are the central focus of Lusztig's conjecture. For suitably large $p$ and appropriate $\lambda$, the conjecture proposes explicit values for these coefficients.

\begin{Conjecture}[Lusztig's Conjecture]\index{Lusztig's Conjecture}

Assume that $p \geq h$, where $h$ is the Coxeter number\index{Coxeter number} of the Weyl group $W$. Let $w \in W_p$, where $W_p$ is the affine Weyl group\index{affine Weyl group}, and suppose $w$ satisfies the Jantzen condition\index{Jantzen condition}. If $-w \cdot 0$ is dominant\index{dominant}, then the character $\chi_p(-w \cdot 0)$ of the corresponding simple module is given by

\[
\chi_p(-w \cdot 0) = \sum_{\substack{y \leq w \\ -y \cdot 0\ \text{dominant}}} (-1)^{\ell(w) - \ell(y)} P_{y,w}(1) \chi(-y \cdot 0),
\]

where $P_{y,w}(q)$ are the Kazhdan–Lusztig polynomials and $\ell$ denotes the length function on $W_p$. The sum is taken over those $y \in W_p$ such that $-y \cdot 0$ is dominant.

\end{Conjecture}

\chapter{Modular Lie Algebras}

The previous chapter demonstrated that Lie algebras over a base field of prime characteristic exhibit fundamentally different properties compared to those over an algebraically closed field or field of characteristic zero. Lie algebras defined over a field of prime characteristic are referred to as \textit{modular Lie algebras}\index{modular Lie algebras}. As previously observed, many foundational theorems in Lie theory—such as Lie’s theorem, Cartan’s criterion, and Weyl’s theorem—do not apply in the modular setting. To study modular Lie algebras effectively, new mathematical tools must be introduced. These tools reveal structural properties that are unique to the modular context.

\section{$p$-Mappings and Restricted Lie Algebras} \label{Modular Lie Algebras}

The previous chapter indicates that Lie algebras over a field \( F \) with characteristic \( p \) require completely different tools. Therefore, a \textit{$p$-mapping} \( [p] \), defined by \( x \mapsto x^{[p]} \), is introduced. However, before introducing the \( p \)-mapping, it is necessary to establish general intuition. This insight is gained by considering the \( p \)-th power map\index{p-th power map}, defined by \( x \mapsto x^p \), in a Lie algebra \( L \) equipped with its commutator. The discussion refers to \cite[Chapter 2]{SF} and \cite[Section 7, Chapter 5]{N}.

\subsection{Basic Settings for $p$-Mappings}\label{Setpmap}

Recall that in \ref{algLiealg}, any associative $F$-algebra $A$ can be made into a Lie algebra $A_L$ with its commutator.

The basic setting in this section is that the base field $F$ has prime characteristic $p$.

\begin{Remark}\label{remp}
\item[(i)]\label{remi}
Let $A$ be an associative $F$-algebra where $\operatorname{char}F=p$.

First, let $x, y \in A$ and consider $(\ad x)^m(y)$ for $m \in \mathbb{N}$. Denote $L_a$ and $R_a$ as the left and right multiplication by $a \in A$, i.e., $L_a(b) = ab$ and $R_a(b) = ba$. Then
$$(\ad a)(b) = [a, b] = ab - ba = L_a(b) - R_a(b) = (L_a - R_a)(b).$$
Thus, 
$$ (\ad a) = L_a - R_a, \quad \forall a \in A. $$

Hence, 
$$ (\ad x)^m(y) = (L_x - R_x)^m(y) = \sum_{i=0}^{m} (-1)^{m-i} \binom{m}{i} L_x^i \cdot R_x^{m-i}(y) = \sum_{i=0}^{m} (-1)^{m-i} \binom{m}{i} x^i y x^{m-i}. $$

\begin{equation}\label{adxmy}
(\ad x)^m(y) = \sum_{i=0}^{m} (-1)^{m-i} \binom{m}{i} x^i y x^{m-i}.
\end{equation}

Substitute $m = p$ in the equation. Note that all of the middle terms vanish, since the binomial coefficients $\binom{p}{i} \equiv 0 \pmod{p}$ for $i = 1, \cdots, p-1$. Thus, 
$$(\ad x)^p(y) = x^p y - y x^p = (\ad x^p)(y).$$
Hence,
$$ \ad x^p = (\ad x)^p \quad \forall x \in A. $$

The identity indicates that the power of $p$ on the adjoint representation of $x$ is the adjoint representation of $x^p$. This implies that the $p$-th power map $x \mapsto x^p$ has a structural interrelation with $A$.

\item[(ii)]

Clearly, $(\alpha x)^p = \alpha^p x^p \;\; \forall\; \alpha \in F, \; x \in A$ due to $A$ being an associative $F$-algebra.

\item[(iii)]

Finally, consider $(x + y)^p$ for any $x, y \in A$. To expand this term by term, consider $(aX + b)^p \in A[X]$ where $X$ is an indeterminate over $A$. The polynomial $(aX + b)^p$ is written as 
$$
(aX + b)^p = a^p X^p + b^p + \sum_{i=1}^{p-1} s_i(a, b) X^i 
$$ 
where $s_i(a, b) \in A$.

\begin{equation}\label{axbp}
(aX + b)^p = a^p X^p + b^p + \sum_{i=1}^{p-1} s_i(a, b) X^i \quad \text{where } s_i(a, b) \in A.
\end{equation}

The expression denotes only $X^p$ and constant terms explicitly but leaves the remainder terms by denoting coefficients as $s_i(a, b), 1 \leq i \leq p-1$. 

By substituting $X = 1$ in this equation in order to find $(x + y)^p$ for any $x, y \in A$, we get
\begin{equation}\label{abab}
(a + b)^p = a^p + b^p + \sum_{i=1}^{p-1} s_i(a, b), \quad \forall a, b \in A.
\end{equation}

Note that each value $s_i(a, b)$ is not expressed explicitly. The values of $s_i(a, b)$ are uncertain. Thus, differentiate equation \ref{axbp} with respect to $X$. Then, we obtain
\begin{equation}\label{dif}
p a (aX + b)^{p-1} = p a^p X^{p-1} + \sum_{i=1}^{p-1} i s_i(a, b) X^{i-1}.
\end{equation}

Note that the left-hand side $p a (aX + b)^{p-1}$ is $0$ since the term has $p$ as a factor. Instead of writing the term as $0$, we can modify it using the following technique to express it as $(\ad (aX + b))^{p-1}(a)$ in equation \ref{adxmy}:
$$
p a (aX + b)^{p-1} = \sum_{i=0}^{p-1} a (aX + b)^{p-1} = \sum_{i=0}^{p-1} (aX + b)^i a (aX + b)^{p-1-i}.
$$

$ p a (aX + b)^{p-1} $ is obtained by simply adding $ a (aX + b)^{p-1} $ with $p$ times and $(aX + b)^{p-1} = (aX + b)^i (aX + b)^{p-1-i}$ for all $i = 1, \cdots, p-1$.

Note that $p a^p X^{p-1} = 0$.

Therefore, equation \ref{dif} becomes
\begin{equation}\label{sumaxb1}
\sum_{i=0}^{p-1} (aX + b)^i a (aX + b)^{p-1-i} = \sum_{i=1}^{p-1} i s_i(a, b) X^{i-1}.
\end{equation}

Note that $\binom{p-1}{i} = (-1)^i \; (\text{mod } p)$ holds for $i = 1, \dots, p-1$. Thus, equation \ref{sumaxb1} becomes

$$
\sum_{i=0}^{p-1} (-1)^i \binom{p-1}{i} (aX + b)^i a (aX + b)^{p-1-i} = \sum_{i=1}^{p-1} i s_i(a, b) X^{i-1}.
$$

By substituting $m = p-1$, $x = aX + b$, and $y = a$ into equation \ref{adxmy}, we get

$$
(\ad\; (aX + b))^{p-1}(a) = \sum_{i=0}^{p-1} (-1)^i \binom{p-1}{i} (aX + b)^i a (aX + b)^{p-1-i} = \sum_{i=1}^{p-1} i s_i(a, b) X^{i-1}.
$$

The result of this equation is

\begin{equation}\label{adaxb}
(\ad\; (aX + b))^{p-1}(a) = \sum_{i=1}^{p-1} i s_i(a, b) X^{i-1}.
\end{equation}

This shows that $i s_i(a, b)$ is the coefficient of $X^{i-1}$ in the polynomial $(\ad (aX + b))^{p-1}(a) \in A[X]$. This implies that the previously unknown values $s_i(a, b)$ can be determined from the above equation \ref{adaxb}.

In conclusion, 
$$
(a + b)^p = a^p + b^p + \sum_{i=1}^{p-1} s_i(a, b), \quad \forall a, b \in A
$$ 
where $s_i(a, b)$ is the value derived from $(\ad\; (aX + b))^{p-1}(a) = \sum_{i=1}^{p-1} i s_i(a, b) X^{i-1}$.

\end{Remark}

\textit{Remark} \ref{remi} is in the setting of associative $F$-algebra $A$. However, the basic setting of $p$-mapping is modular Lie algebras. Therefore, it is necessary to change the setting from the polynomial ring $A[X]$ to the Lie algebra polynomial ring $L \otimes_F F[X]$.

\begin{Remark}

First, it is possible to identify $A \otimes_F F[X]$ and $A[X]$ with the unique isomorphism $f: A \otimes_F F[X] \rightarrow A[X]$ such that $f(a \otimes X) = aX, \;\; \forall a \in A$. Then, the condition for $s_i(a, b)$ in equation \ref{adaxb} is characterized in $A \otimes_F F[X]$ as

\begin{equation}
(\ad \; (aX \otimes 1 + b \otimes 1))^{p-1}(a \otimes 1) = \sum_{i=1}^{p-1} i s_i(a, b) \otimes X^{i-1}.
\end{equation}

Next, given a Lie algebra $L$, $L \otimes_F F[X]$ has a Lie structure defined by
$$[a \otimes r, b \otimes s] = [a, b] \otimes rs \quad a, b \in L, \; r, s \in F[X].$$

Then, every $h \in L \otimes_F F[X]$ can be uniquely expressed as
$$h = \sum_{i=0}^{n} a_i \otimes X^i, \quad a_i \in L.$$

By the entire procedure, the $p$-th power map $\mathfrak{p}$ transforms into a map in the modular Lie algebra structure setting.

\end{Remark}

\subsection{Definition of $p$-Mappings}

In the previous section, given a Lie algebra $L$ with Lie multiplication $[X,Y]=XY-YX$ for all $X,Y \in L$, the $p$-th power map $\mathfrak{p}: L \rightarrow L$ defined by $x \mapsto x^p$ provides general intuition on how the $p$-mapping should be defined in the modular Lie algebra setting. The previous section justifies that the map $\mathfrak{p}: x \mapsto x^p$ is well-constructed within the modular Lie algebra structure.

With this insight, the $p$-mapping in the modular Lie algebra setting is defined below, as a generalized definition of the map $\mathfrak{p}: x \mapsto x^p$. Thus, the $p$-mapping holds all the discussions from the previous section and is defined in the same setting as the $p$-th power map $\mathfrak{p}$, except the $p$-mapping is for any modular Lie algebra with any Lie multiplication defined.

\begin{Definition}

Let $L$ be a Lie algebra over $F$ of $\operatorname{char} p$. A mapping $[p]: L \rightarrow L$ defined by $x \mapsto x^{[p]}$ is called a \textit{$p$-mapping}\index{p-mapping} if:

\begin{itemize}\label{pmap1}
    \item[1.] $\ad a^{[p]} = (\ad a)^p$ \; for all $a \in L$.
    
    \item[2.] $(\alpha a)^{[p]} = \alpha^p a^{[p]}$ \; for all $\alpha \in F$ and $a \in L$.
    
    \item[3.] $(a + b)^{[p]} = a^{[p]} + b^{[p]} + \sum^{p-1}_{i=1} s_i(a, b)$ where $(\ad(a \otimes X + b \otimes 1))^{p-1}(a \otimes 1) = \sum^{p-1}_{i=1} i s_i(a, b) \otimes X^{i-1}$ in $L \otimes_F F[X]$ for all $a, b \in L$.
\end{itemize}

\end{Definition}

\begin{Example}\label{glpmap}

In $\gl(n,F)$, the $p$-th power map $\mathfrak{p}: \gl(n,F) \rightarrow \gl(n,F)$ such that $x \mapsto x^p$ is a $p$-mapping. This is clear because the Lie multiplication is defined as $[X,Y] = XY - YX$ for any $X, Y \in \gl(n,F)$, the same as in the previous section \ref{Setpmap}.

\end{Example}

\begin{Definition}
Given a Lie algebra $L$ of characteristic $p$, a map $f: L \rightarrow L$ such that 
\[
f(a+b) = f(a) + f(b) \quad \text{for all } a,b \in L \quad \text{and} \quad f(\alpha a) = \alpha^p f(a) \quad \text{for all } \alpha \in F, \; a \in L
\]
is called a \textit{$p$-semi-linear mapping}\index{p-semi-linear mapping}.

\end{Definition}

\subsection{Definition of Restricted Lie Algebras}

\begin{Definition}

Let $L$ be a Lie algebra over $F$ with $\operatorname{char}p$. If $L$ has a $p$-mapping $[p]$, then $L$ is called a \textit{restricted Lie algebra}\index{restricted Lie algebra}.

\end{Definition}

\begin{Example}
Let $L = \langle x, y, z \; |\; [x,y] = z, [y,z] = 0, [z,x] = 0 \rangle$. Then the $p$-th power map is a $p$-mapping that satisfies $x^{[p]} = y^{[p]} = z^{[p]} = 0$.

In this case, the Lie subalgebra $\langle z \rangle$ is the centre of $L$. This example refers to \cite{BD}.
\end{Example}

\subsection{Uniqueness of $p$-mapping in Restricted Lie Algebras}

Note that $p$-mappings in a restricted Lie algebra are not necessarily unique. However, any 2-dimensional non-abelian Lie algebra has a unique $p$-mapping.

\begin{Corollary}

A 2-dimensional non-abelian Lie algebra has a unique $p$-mapping.

\end{Corollary}

\begin{proof}

Let $L := \langle h, x \;|\; [h,x]=x \rangle$. Then $L$ induces a unique $p$-mapping by defining 
\[
z^{[p]} = (\alpha h + \beta x)^{[p]} := \alpha^p h + \alpha^{p-1} \beta x.
\]

Let's check this map is a $p$-mapping. Let $L$ be the Lie algebra generated by $h = e_{11}$ and $x = e_{12}$.
Then $L = Fh \oplus Fx$ satisfying $[h, x] = e_{11}e_{12} - e_{12}e_{11} = e_{12} = x$.

Let 
\[
z = \begin{pmatrix}
\alpha & \beta \\
0 & 0
\end{pmatrix} = \alpha h + \beta x \in L
\]
for some $\alpha, \beta \in F$. Clearly, $z^{[p]} = z^p$ by the following.

Use induction to prove that $z^n = \alpha^n h + \alpha^{n-1} \beta x$ for any $n \in \mathbb{N}$. It is clear that 
\[
z = \begin{pmatrix}
\alpha^1 & \alpha^0 \beta \\
0 & 0
\end{pmatrix}.
\]
Assume 
\[
z^k = \begin{pmatrix}
\alpha^k & \alpha^{k-1} \beta \\
0 & 0
\end{pmatrix} \quad \text{for all} \quad 1 \leq k \leq n-1.
\]
Then 
\[
z^n = \begin{pmatrix}
\alpha & \beta \\
0 & 0
\end{pmatrix} \begin{pmatrix}
\alpha^{n-1} & \alpha^{n-2} \beta \\
0 & 0
\end{pmatrix} = \begin{pmatrix}
\alpha^n & \alpha^{n-1} \beta \\
0 & 0
\end{pmatrix} = \alpha^n h + \alpha^{n-1} \beta x.
\]
By induction, $z^p = \alpha^p h + \alpha^{p-1} \beta x$. Therefore, $z^{[p]} = z^p$ for all $z \in L$.

Since the map is the $p$-th power map $\mathfrak{p}$, we can conclude that this is a $p$-mapping. Since $\mathfrak{p}$ is the only possible $p$-mapping, this map is the unique $p$-mapping for $L$.
\end{proof}

Moreover, there is a relationship between the uniqueness of the \( p \)-mapping and the Killing form.

\begin{Theorem}

Let \( L \) be a finite-dimensional Lie algebra of characteristic \( p > 0 \) with a non-degenerate Killing form. Then, there exists a unique \( p \)-mapping in \( L \).

\end{Theorem}

\begin{proof}

For the proof, see \cite[Theorem 5.7.11]{N}.
\end{proof}

\section{Restricted Special Linear Lie Algebras}

Since the special linear Lie algebra $\mathfrak{sl}(n, F)$ is a Lie subalgebra of $\gl(n, F)$, which has a $p$-mapping defined by $x \mapsto x^p$, it is natural to ask whether $\mathfrak{sl}(n, F)$ also admits a $p$-mapping.

\begin{Corollary}\label{slnFpmap}
The special linear Lie algebra $\mathfrak{sl}(n,F)$ has a $p$-mapping.
\end{Corollary}

\begin{proof}

In \ref{glpmap}, the $p$-th power map $\mathfrak{p}: \gl(n,F) \rightarrow \gl(n,F)$, defined by $x \mapsto x^p$, is a $p$-mapping on $\gl(n,F)$. Since $\mathfrak{sl}(n,F)$ is the special linear Lie subalgebra of $\gl(n,F)$, it is natural to check whether the restriction of the $p$-th power map of $\gl(n,F)$, 
$$\text{i.e.,}\quad \mathfrak{p}|_{\mathfrak{sl}(n,F)}: \mathfrak{sl}(n,F) \rightarrow \gl(n,F), \quad \text{defined by}\quad x \mapsto x^p \quad\text{for all}\;\; x \in \mathfrak{sl}(n,F),$$
is also the $p$-th power map on $\mathfrak{sl}(n,F)$. In other words, we want to verify that the map $\mathfrak{p}|_{\mathfrak{sl}(n,F)}$ satisfies
$$\mathfrak{p}: \mathfrak{sl}(n,F) \rightarrow \mathfrak{sl}(n,F),$$
i.e., that 
$$x^p \in \mathfrak{sl}(n,F)\quad \text{for all}\;\; x \in \mathfrak{sl}(n,F).$$
This is equivalent to checking that $\tr(x^p) = 0$ for all $x \in \mathfrak{sl}(n,F)$.

Let $x \in \mathfrak{sl}(n,F)$. Then $\tr(x) = 0$. The characteristic polynomial of $x$ is
$$\chi_x(\lambda) = \lambda^n - \tr(x)\lambda^{n-1} + a_{n-2}\lambda^{n-2} + \cdots + a_0 = \lambda^n + a_{n-2}\lambda^{n-2} + \cdots + a_0 = 0.$$
By the Cayley-Hamilton theorem, we have
$$\chi_x(x) = x^n + a_{n-2}x^{n-2} + \cdots + a_0 = 0.$$
Taking the $p$-th power of this equation, we obtain
$$\{\chi_x(x)\}^p = (x^p)^n + {a_{n-2}}^p(x^p)^{n-2} + \cdots + {a_0}^p = 0.$$
Thus,
$$\chi_{x^p}(x^p) = (x^p)^n + {a_{n-2}}^p(x^p)^{n-2} + \cdots + {a_0}^p = 0.$$
This implies that the characteristic polynomial of $x^p$ is
$$\chi_{x^p}(\lambda) = \lambda^n - {a_{n-2}}^p\lambda^{n-2} + \cdots + {a_0}^p.$$
Note that the coefficient of $\lambda^{n-1}$ in the characteristic polynomial of $x^p$ is zero. Therefore, $\tr(x^p) = 0$, which means that $x^p \in \mathfrak{sl}(n,F)$. This shows that the restriction map $\mathfrak{p}|_{\mathfrak{sl}(n,F)}$ is indeed the $p$-th power map on $\mathfrak{sl}(n,F)$, and hence it is a $p$-mapping on $\mathfrak{sl}(n,F)$.
\end{proof}

\begin{Remark}\label{sl2pmap}

\item[1.] The restricted map of $\mathfrak{p}$ on $\mathfrak{sl}(n,F)$, denoted $\mathfrak{p}|_{\mathfrak{sl}(n,F)}$, is a $p$-mapping of $\mathfrak{sl}(n,F)$.

\item[2.] $\mathfrak{sl}(2,F)$ has the $p$-mapping $\mathfrak{p}: x \mapsto x^p$, so $\mathfrak{sl}(2,F)$ is a restricted Lie algebra. In fact, $\mathfrak{sl}(2,F)$ with the usual basis
$$ e = \begin{pmatrix}
 0 & 1 \\
 0 & 0 \\
\end{pmatrix}, \quad
 f = \begin{pmatrix}
 0 & 0 \\
 1 & 0 \\
\end{pmatrix}, \quad \text{and} \quad
 h = \begin{pmatrix}
 1 & 0 \\
 0 & -1 \\
\end{pmatrix} $$
satisfies $e^{[p]} = 0$, $f^{[p]} = 0$, and $h^{[p]} = h$.

\end{Remark}

\subsection{$\fsl_2$: Not a Restricted Lie Algebra}

From Corollary \ref{sl2pmap}, $\mathfrak{sl}(2, F)$ admits a $p$-mapping. It is therefore natural to ask whether $\mathfrak{sl}(2, F)$ also admits a $p$-mapping. Recall that $\operatorname{char} F > 2$ in the case of $\mathfrak{sl}(2, F)$ and $\operatorname{char} F = 2$ in the case of $\mathfrak{fsl}(2, F)$.

However, it turns out that $\mathfrak{fsl}(2, F)$ does \emph{not} admit a $p$-mapping. As a result, $\mathfrak{fsl}(2, F)$ is not a restricted Lie algebra.

\begin{Theorem}

$\mathfrak{fsl}(2, F)$ has no $p$-mapping, so it is not a restricted Lie algebra.

\end{Theorem}

\begin{proof}

Let us assume that \( \mathfrak{fsl}(2,F) \) has a \( p \)-mapping \( [p] : \mathfrak{fsl}(2,F) \to \mathfrak{fsl}(2,F) \) defined by \( x \mapsto x^{[p]} \). By the first condition of the \( p \)-mapping definition in \ref{pmap1}, one of the generators, \( e \), should satisfy \( \ad e^{[2]} = (\ad e)^2 \). This implies that there exists an element \( \alpha = e^{[2]} \in \mathfrak{fsl}(2,F) \) such that \( \ad \alpha = (\ad e)^2 \). Recall that \( (\ad e)^2(x) = [e, [e, x]] \) for all \( x \in \mathfrak{fsl}(2,F) \).

Then, applying \( (\ad e)^2 \) on each generator, we get:

\[
(\ad e)^2(e) = [e, [e, e]] = [e, 0] = 0
\]

\[
(\ad e)^2(h) = [e, [e, h]] = [e, -e] = 0
\]

\[
(\ad e)^2(f) = [e, [e, f]] = [e, h] = -e
\]

Therefore, \( \alpha \) should satisfy the following three equations:

\begin{equation}\label{alpe}
(\ad \alpha)(e) = [\alpha, e] = 0
\end{equation}

\begin{equation}\label{alph}
(\ad \alpha)(h) = [\alpha, h] = 0
\end{equation}

\begin{equation}\label{alpf}
(\ad \alpha)(f) = [\alpha, f] = -e
\end{equation}

Recall that \( \mathfrak{fsl}(2,F) \) is 3-dimensional and has the basis \( \{e, f, h\} \). From \ref{alpe}, \( \alpha = ke \) for some \( k \in F \). Since \( F \) has characteristic 2, \( \alpha \) must be either \( e \) or 0. Clearly, \( \alpha \) cannot be 0, because if it were, then \( (\ad \alpha)(f) = [0, f] = 0 \), which contradicts \ref{alpf}. Let \( \alpha = e \). Then, \( (\ad \alpha)(h) = [e, h] = -e \neq 0 \), which contradicts \ref{alph}. Therefore, there is no such \( \alpha \) in \( \mathfrak{fsl}(2,F) \). By contradiction, \( \mathfrak{fsl}(2,F) \) does not have a \( p \)-mapping.
\end{proof}

\chapter{Conclusions}

Lie algebras have many theorems that help to construct Lie theory. What about
modular Lie algebras? Since a modular Lie algebra is also a Lie algebra, it shares a
similar context with Lie algebras and, in general, Lie theory. However, as we’ve seen
in this main project, Lie theory theorems do not always hold. From this perspective,
Lie theory and modular Lie theory are different.

For example, the Weyl module and Lusztig’s conjecture are introduced to study
Weyl’s theorem, the Weyl character formula, and, more generally, the characters of
representations in prime characteristic fields. The invaluable insight gained from the
limitation of Weyl’s theorem in a prime characteristic field is that reduction modulo
p changes the results of calculations and the properties of Lie algebras.

Another key point is that a semisimple Lie algebra plays a crucial role in Lie
theorems in a prime characteristic field, such as the preservation of the Jordan decom-
position. Semisimple Lie algebras are related to solvable Lie algebras. Furthermore,
their representations are also an important topic in modular Lie theory.

Since the definition of the $p$-mapping is given, we want to understand how the
$p$-mapping is used in modular Lie theory. The $p$-mapping seems quite similar to
the Frobenius map in field theory. With this $p$-mapping, modular Lie theory can be
expanded to include the homology of Lie algebras.

Modular Lie algebras and their representations seem similar to classical Lie theory
but are not exactly the same. In fact, the representations of modular Lie algebras have
interesting properties. For example, Weyl’s theorem, the Jordan decomposition, and
irreducible representations of Lie algebras in this main project are related to modular
Lie algebra representations. It is well-known that characters reveal many properties
of their representations, so it would be interesting to see the properties of characters
of modular Lie algebras (or restricted Lie algebras). Studying $p$-adic representations
would also provide deeper insights into modular Lie algebra representations. Many
branches of modular Lie theory would offer intriguing motivations to study Lie alge-
bras.

\bibliography{BiblioMainProj}{}
\bibliographystyle{plain}


\newpage
\printindex

\end{document}